\documentclass[12pt]{amsart}
\usepackage{amsmath,latexsym,amscd,amsbsy,amssymb,amsfonts,amsthm,fleqn,leqno,
euscript, graphicx, texdraw, pb-diagram}
\usepackage{mathrsfs}
\usepackage[matrix,arrow,curve]{xy}

\newtheorem{thm}{Theorem}[section]
\newtheorem{pro}[thm]{Proposition}
\newtheorem{lem}[thm]{Lemma}

\newtheorem{cor}[thm]{Corollary}

\def\leukfrac#1/#2{\leavevmode
               \kern.1em
                \raise.9ex\hbox{\the\scriptfont0 ${}_#1$}
                \hskip -1pt\kern-.1em
                /\kern-.15em\lower.10ex\hbox{\the\scriptfont0 ${}_#2$}}
\newtheorem{qu}[thm]{Question}

\theoremstyle{definition}

\theoremstyle{remark}
\newtheorem{claim}{Claim}

\makeatletter
\def\Int{\mathop{\operator@font Int}\nolimits}
\makeatother

\begin{document}

\title[Homological characterizations of $Q$-manifolds and $l_2$-manifolds]
{Homological characterizations of $Q$-manifolds and $l_2$-manifolds}

\author{Alexandre Karassev}
\address{Department of Computer Science and Mathematics, Nipissing University,
100 College Drive, P.O. Box 5002, North Bay, ON, P1B 8L7, Canada}
\email{alexandk@nipissingu.ca}
\thanks{The first author was partially supported by NSERC Grant 257231-14}

\author{Vesko  Valov}
\address{Department of Computer Science and Mathematics, Nipissing University,
100 College Drive, P.O. Box 5002, North Bay, ON, P1B 8L7, Canada}
\email{veskov@nipissingu.ca}
\thanks{The second author was partially supported by NSERC Grant 261914-19}



%%%%%%%%%% End topmatter %%%%%%%%%%%%%%%%%%%%%

\begin{abstract}
We investigate to what extend the density of $Z_n$-maps in the characterization of $Q$-manifolds, and the density of maps
$f\in C(\mathbb N\times Q,X)$ having discrete images in the $l_2$-manifolds characterization can be weakened to the density of homological $Z_n$-maps and homological $Z$-maps, respectively. As a result, we obtain homological characterizations of $Q$-manifolds and $l_2$-manifolds.
\end{abstract}

\maketitle

\markboth{}{Homology characterizations}

%%%%%%%%%%%%%%%%%%%%%%%%%%%%%%%%%%%%%%%%%%%%%%%%%%%%%%%%%%%%%%%%
%%%%%%%%%%%%%%%%%%%%%%%%%%%%%%%%%%%%%%%%%%%%%%%%%%%%%%%%%%%%%%%%

%%%%%%%%%%%%%%%%%%%%%%%%%%%%%%%%%%%%%%%%%%%%%%%%%%%%%%%%%%%%%%%%%%%%%%

\section{Introduction and preliminary results}
By a space we always mean a complete separable metric space without isolated points.

The well-known Toru\'{n}czyk's fundamental characterizations of manifolds modeled on $Q=[-1,1]^\infty$ and $l_2$ states that a locally compact separable $ANR$-space $X$ is a $Q$-manifold if and only if $X$ satisfies the disjoint $n$-disks property for every $n$, \cite{to1}. Equivalently, for every $n$ the function space $C(\mathbb B^n,X)$ contains a dense set of $Z_n$-maps. Similarly, a complete separable $ANR$-space is $l_2$-manifold iff
$X$ has the discrete approximations property: $C(\mathbb N\times Q,X)$ contains a dense set of maps $f$ such that the family $\{f(\{n\}\times Q):n\in\mathbb N\}$ is discrete in $X$, see \cite{to3}. Here, both function spaces $C(Q,X)$ and $C(\mathbb N\times Q,X)$ are equipped with the limitation topology, and $\mathbb B^n$ is the $n$-dimensional ball. Daverman and Walsh \cite{dw} (see also \cite{lw}) refine Toru\'{n}czyk's $Q$-manifolds characterization by combining the disjoint $2$-disks property and the disjoint \v{C}ech carriers property (the latter property means that \v{C}ech homology elements can be made disjoint).
Banakh and Repov\v{s} \cite{br} observed that the disjoint \v{C}ech carriers property in Daverman-Wash's characterization can be replaced by the following one: $C(K,X)$ contains a dense set of homological $Z$-maps for every compact polyhedron $K$. Bowers \cite{bo} provided an $l_2$-version of Daverman-Walsh's result for spaces having nice $ANR$ local compactifications.

In the present paper we investigate to what extend the density of $Z_n$ maps in the characterization of $Q$-manifolds, and the density of maps
$f\in C(\mathbb N\times Q,X)$ having discrete images in the $l_2$-manifolds characterization can be weakened to the density of homological $Z_n$-maps and homological $Z$-maps, respectively.
In Section 2 we establish relations between density
of homological $Z_n$-maps in $C(\mathbb B^n,X)$ and the disjoint $n$-carriers property. As a corollary we obtain a proof of Banach-Repov\v{s} result mentioned above, and obtain a characterization of spaces having a nice $ANR$ local compactification that is a $Q$-manifold. We also discuss the question whether for any locally compact $ANR$-space $X$ we have that $X\times\mathbb B^1$ is a $Q$-manifold if and only if for any $n$ the space $C(\mathbb B^n,X)$ contains a dense set of homological $Z_n$-maps. This can be compared to Daverman-Walsh's result \cite{dw} that a locally compact $ANR$-space has the disjoint \v{C}ech carriers property if and only if $X\times\mathbb B^2$ is a $Q$-manifold.

In Section 3 we provide a homological characterizations of $l_2$-manifolds. Two cases are considered, the case of the boundary set setting (when the spaces under consideration have ``nice" $ANR$ local compactifications) and the general case of $ANR$-spaces. Among $ANR$-spaces $l_2$-manifolds are exactly the spaces $X$ having the discrete $2$-cells property such that $C(\mathbb N\times Q,X)$ contains a dense $G_\delta$-set of homological $Z_\infty$-maps (see Corollary 3.6). A combination of the disjoint disks property and the existence of a dense $G_\delta$-set in $C(\mathbb N\times\mathbb B^n,X)$ of homological $Z_n$-maps for every $n$ characterizes $l_2$-manifolds in the boundary set setting (Corollary 3.2).

 %If not stated otherwise, all homology groups are singular.
%For a compact space $K$ and a metric space $X$ we denote by $C(K,X)$ the space of all continuous maps from $K$ into $X$ topologized either by the compact-open topology or by the limitation topology. In case $K$ is compact both topologies on $C(K,X)$ coincide, see \cite{to3}.
%with the uniform convergence topology generated by a compatible metric on $X$.
The singular and \v{C}ech homology groups are denoted, respectively, by $H_k$ and $\check{H}_k$.
If $V\subset U$ are open subsets of $X$ and $z\in \check{H}_q(U,V)$ for some integer $q\geq 0$,  a compact pair $(C,\partial C)\subset (U,V)$ is said to be a {\em \v{C}ech carrier} \cite{dw} for $z$  provided $z\in i_*(\check{H}_q(C,\partial C))$, where $i_*:\check{H}_q(C,\partial C)\to \check{H}_q(U,V)$ is the inclusion-induced homomorphism. A {\em singular carrier} of an element $z\in H_q(U,V)$ is defined in a similar way.
%Since singular homology has compact supports, every $z\in H_q(U,V)$ has a carrier.
Following \cite{dw}, we say that a space $X$ has the {\em disjoint $n$-carriers property} (br., $\rm{DC^nP}$) provided for any open in $X$ sets $V_i\subset U_i$, $i=1,2$,  and any elements $z_i\in \check{H}_{q(i)}(U_i,V_i)$ with $0\leq q(i)\leq n$ there are \v{C}ech carriers $(C_i,\partial C_i)\subset (U_i,V_i)$ for $z_i$ such that $C_1\cap C_2=\varnothing$. A space $X$ has the {\em disjoint \v{C}ech carriers property} (br., $\rm{DCP}$) provided $X\in\rm{DC^nP}$ for every $n$.
The disjoint $n$-carriers property is a homological analogue of the well-known {\em disjoint $n$-disks property} (br., $\rm{DD^nP}$): any two maps $\mathbb B^n\to X$ can be approximated by maps with disjoint images.  Recall that a closed set $F\subset X$ is said to be a {\em $Z_n$-set} if the set $C(\mathbb B^n,X\setminus F)$ is dense in $C(\mathbb B^n,X)$. Note that if $X$ is a $LC^{n-1}$-space, then a closed set $F\subset X$ is $Z_n$-set iff for each at most $n$-dimensional metric compactum $Y$ the set $\{f\in C(Y,X):f(Y)\cap F=\varnothing\}$ is dense in $C(Y,X)$, see \cite{b}.
We also say that a map $f:K\to X$, where $K$ is a compactum, is a {\em $Z_n$-map} provided $f(K)$ is a $Z_n$-set in $X$.
According to \cite{to1},
$X\in\rm{DD^nP}$ if and only if the set of all $Z_n$-maps  $\mathbb B^n\to X$ is dense and $G_\delta$ in $C(\mathbb B^n,X)$.

Replacing $\mathbb B^n$ in the above definition of $Z_n$-sets with $Q$, we obtain the definition of $Z_n$-sets and $Z_n$-maps.
It is well known that a set is a $Z$-set iff its a $Z_n$-set for every $n$.
 A closed set $A\subset X$ is a {\em strong $Z$-set} \cite{bbmw} if for every open cover $\mathcal U$ of $X$ and a sequence of maps $\{f_i\}\subset C(Q,X)$, there is a sequence $\{g_i\}\subset C(Q,X)$ such that each $g_i$ is $\mathcal U$-close to $f_i$ and $\overline{\bigcup_{i\geq 1}g_i(Q)}\cap A=\varnothing$.

%Our first result describes $Z_n$-subsets of $LC^n$-spaces.
Recall that a space $X$ is $LC^n$ if for every $x\in X$ and its neighborhood $U$ in $X$ there is another neighborhood $V$ of $x$ such that $V\stackrel{m}{\hookrightarrow}U$ for all $m\leq n$ (here $V\stackrel{m}{\hookrightarrow}U$ means that $V\subset U$ and every map from the $m$-dimensional sphere $\mathbb S^m$ into $V$ can be extended to a map $\mathbb B^{m+1}$ to $U$). We also say that a set $A\subset X$ is
$k-LCC$ in $X$ if for every point $x\in A$ and its neighborhood $U$ in $X$ there exists another neighborhood  $V$ of $x$ with $V\subset U$ and $V\setminus A\stackrel{k}{\hookrightarrow}U\setminus A$. If $A$ is $k-LCC$ in $X$ for all $k\leq n$, then $A$ is said to be $LCC^n$ in $X$.

There are homological analogues of $Z_n$-sets and $Z_n$-maps. A closed set $F\subset X$ is called a {\em homological $Z_n$-set in $X$} if the singular homology groups $H_k(U,U\setminus F)$ are trivial for all open sets $U\subset X$ and all $k\leq n$, see \cite{bck}. It can be shown that every homological $Z_n$-set in $X$ is nowhere dense. The homological $Z_n$-property is finitely additive and hereditary with respect to closed subsets \cite{bck}.
A map $f:K\to X$  is a {\em homological $Z_n$-map} provided the image $f(K)$ is a homological $Z_n$-set in $X$. Homological $Z_\infty$-sets were considered in \cite{dw} under the name sets of infinite codimension.

%\begin{pro}\cite{bck}\label{homotopical}
%A closed subset $A$ of an $LC^1$-space is locally $n$-negligible, where $n\geq 2$, if and only if $A$ is locally a 2-negligible homological $Z_n$-set in $X$.
%\end{pro}

Combining \cite[Corollary 3.3]{to2} and \cite[Theorem 2.1]{bck}, we have the following:
\begin{pro}\label{z}
Let $X$ be an $LC^{n}$-space with $n\geq 2$. Then, a closed subset $A$ of $X$ is a $Z_n$-set in $X$ provided $A$ is an $LCC^1$ homological $Z_n$-set in $X$. Equivalently, $A$ is a $Z_n$-set iff it is a $Z_2$-set and a homological $Z_n$-set.
In particular, $Z_n$-sets in $LC^n$-spaces are homological $Z_n$-sets.
\end{pro}

\section{Homological $Z_n$-maps and $Q$-manifolds}
  It is well known that for every pair $(X,A)$ with $A\subset X$ there is a natural homomorphism
$T_{X,A}:H_*(X,A)\to\check{H}_*(X,A)$. Recall that a space $X$ is called homologically locally connected up to dimension $n$ if for any point $x$ in $X$ and any neighbourhood $U$ of $x$ there exists a neighbourhood $V$ of $x$ such that $V\subset U$ and the inclusion-induced homomorphism $H_k (V) \to H_k (U)$ is trivial for all $k\le n.$ If both $X$ and $A$ are homologically locally connected with respect to the singular homology up to dimension $n$, then
$T_{X,A}:H_k(X,A)\to\check{H}_k(X,A)$ is an isomorphism for all $k\leq n$, see \cite{m}. In particular, this is true if $X$ and $A$ are $LC^{n}$-spaces.
We say that an element $z\in H_k(U,V)$ has a {\em homological \v{C}ech $Z_n$-carrier} provided $z$ has a \v{C}ech
carrier $(C_z,\partial C_z)\subset (U,V)$ of $z$ such that $C_z$ is a homological $Z_n$-set in $X$.

The following result is well-known.
\begin{lem}
Let $X$ be an $LC^{n}$-space and $V\subset U$ open in $X$. Then $H_k(U,V)$ is countable for all $k\leq n$.
\end{lem}%Everywhere below $X$ is a locally compact separable metric $LC^{n-1}$-space with $\dim X=n$.

%\begin{proof}
%Let $M_k$ be a countable dense set in $C(\mathbb B^k,X)$, and $\mathcal H_k(U,V)$ consists of all $z\in H_k(U,V)$ such that $z$ is represented by a singular chain $\sum_{i=1}^{k}m_if_i$ for some integers $m_i$ and maps $f_i\in M_k$. Clearly  $\mathcal H_k(U,V)$ is countable. Let show that
%$\mathcal H_k(U,V)= H_k(U,V)$. Indeed, if $z\in H_k(U,V)$ and $c_z=\sum_{i=1}^{k}m_ig_i$ is a singular representation of $z$ with $\partial c_z\subset V$, then for each $i$ there is $f_i\in M_k$ such that $f_i$ is homotopic to $g_i$ and $f_i(\mathbb B^k)\subset U$, $f_i(\partial\mathbb B^k)\subset V$ (this can be done because $X\in LC^n$ and $M_k$ is dense in $C(\mathbb B^k,X)$). Hence, $c'_z=\sum_{i=1}^{k}m_if_i$ is a singular representation of $z$, which means that
%$z\in \mathcal H_k(U,V)$. Therefore, $H_k(U,V)$ is countable.
%\end{proof}

%We say that a compact pair $(C_z,\partial C_z)\subset (U,V)$ is a singular carrier for $z\in H_k(U,V)$ if $x\in i_*(H_k(C,\partial C))$, where $i_*:H_k(C,\partial C)\to H_k(U,V)$ is the inclusion-induced homomorphism.
\begin{lem}\label{carrier}
Let $X$ be an $LC^n$-space and $(U,V)$ a pair of open sets in $X$. If $(C_z,\partial C_z)\subset (U,V)$ is a singular carrier for some $z\in H_k(U,V)$ with $k\leq n$, then $(C_z,\partial C_z)$ is also a \v{C}ech carrier for $z$.
\end{lem}
\begin{proof}
Since $X$ is $LC^n$, the homomorphism $T_{U,V}: H_k(U,V)\to\check{H}_k(U,V)$ is an isomorphism. Then the commutative diagram

{ $$
\begin{CD}
H_{k}(C_z,\partial C_z)@>{i_*}>>H_{k}(U,V)\\
@VV{T_{C_z,\partial C_z}}V@VV{T_{U,V}}V\\
\check{H}_{k}(C_z,\partial C_z)@>{i_*}>>\check{H}_{k}(U,V)@.
\end{CD}
$$}\\

implies that $(C_z,\partial C_z)$ is a \v{C}ech carrier for $z$.
\end{proof}

Next lemma is an analogue of \cite[Lemma 3.1]{dw}.
\begin{lem}\label{1}
A closed set $A\subset X$ is a homological $Z_n$-set in $X$ if and only if each $z\in H_m(U,V)$, where $m\leq n$ and $U,V$ are open sets in $X$ with $V\subset U$, has a singular carrier $(C_z,\partial C_z)\subset (U,V)$ such that $C_z\cap A=\varnothing$.
\end{lem}
\begin{proof}
Suppose $A$ is a homological $Z_n$-set and let $z\in H_m(U,V)$ for some $m\leq n$ and open sets $V\subset U$.  Then, we have the commutative diagram,
$$
\xymatrix{
&H_m(U\setminus A)\ar[rr]^{}\ar[d]_{i}&&H_m(U\setminus A,V\setminus A)\ar[rr]^{}\ar[d]_{j}&&
H_{m-1}(V\setminus A\ar[rr]^{})\ar[d]_{p}&&H_{m-1}(U\setminus A)\ar[d]_{l}\\
&H_m(U)\ar[rr]^{}&&H_m(U,V)\ar[rr]^{}&&H_{m-1}(V)\ar[rr]^{}&&H_{m-1}(U)\\
%&\widetilde{H}_k(U;\mathbb Z)\otimes G\ar[rr]^{}&&\widetilde{H}_k(U;G)\ar[rr]^{}&&
%\widetilde{H}_{k-1}(U;\mathbb Z)*G\\
}
$$
where the rows are exact sequences of the pairs $(U\setminus A,V\setminus A)$ and $(U,V)$.
Since $A$ is a homological $Z_n$-set, $i$ is an epimorphism, while $p$ and $l$ are isomorphisms. Hence, by the Four-Lemma, $j$ is  an epimorphism. This means that
 there is $z'\in H_m(U\setminus A,V\setminus A)$ with $j(z')=z$. So, $z$ has a singular carrier $(C_z,\partial C_z)\subset (U\setminus A,V\setminus A)$. %Finally, by Lemma \ref{carrier},  $(C_z,\partial C_z)$ is a \v{C}ech carrier for $z$.

The other implication of Lemma \ref{1} is obvious.
\end{proof}

Let $\mathcal B$ be a finitely additive base for $X$ and $\mathcal H_n=\bigcup\{H_k(U,V):k\leq n{~}\mbox{and}{~}U,V\in\mathcal B\}$. Then
$\mathcal H_n=\bigcup\{H_k(U,V):k\leq n{~}\mbox{and}{~}U,V{~}\mbox{open in}{~}X\}$.
%$\mathcal H_n(U,V)=\{z\in H_k(U,V):k\leq n\}$ and for any $z\in\mathcal H_n(U,V)$ let $(C_z,\partial C_z)\subset (U,V)$ denote a carrier  of $z$.
%If $X$ is $LC^{n}$ each group $H_k(U,V)$, $k\leq n$, is countable and so is the family $\mathcal H_n(U,V)$.

\begin{cor}\label{hom}
Every closed subset of $X$ contained in $X\setminus\bigcup\{C_z:z\in\mathcal H_n\}$ is a homological $Z_{n}$-set in $X$.
\end{cor}

%Following the arguments of \cite[Section 3]{dw} and using Lemma 2.1 instead of Lemma 3.1 from \cite{dw}, one can establish next proposition.
\begin{pro}\label{main}
 Consider the following conditions for an $LC^n$-space:
\begin{itemize}
\item[(1)] $C(\mathbb B^n,X)$ contains a dense set of homological $Z_n$-maps;
\item[(2)] Each $C(\mathbb B^k,X)$, $k\leq n$, contains a dense set of homological $Z_n$-maps;
\item[(3)] Every $z\in H_k(U,V)$, $k\leq n$, has a homological \v{C}ech $Z_n$-carrier $(C_z,\partial C_z)\subset (U,V)$.
\item[(4)] $X\in\rm{DC^nP}$.%If $\{A_i\}$ is a sequence of homological $Z_n$-sets, then every $z\in\mathcal H_n(U,V)$  has a carrier
\end{itemize}
Then $(1)\Rightarrow (2) \Rightarrow (3)\Rightarrow  (4)$.

%$(C_z,\partial C_z)\subset (U,V)$ with $C_z\subset X\setminus\bigcup A_i$;
\end{pro}

\begin{proof}
%Suppose $X$ satisfies condition $(1)$.
$(1)\Rightarrow (2)$:
Suppose $C(\mathbb B^n,X)$ contains a dense set of homological $Z_n$-maps.
For every $k<n$ we embed $\mathbb B^k$ in $\mathbb B^n$ and consider the restriction map $\pi_k^n:C(\mathbb B^n,X)\to C(\mathbb B^k,X)$. The maps
$\pi_k^n$ are open and continuous. Moreover, since $\mathbb B^k$ is a retract of $\mathbb B^n$, $\pi_k^n$ are also surjective. Therefore, if
$M_n\subset C(\mathbb B^n,X)$ is a dense subset consisting of homological $Z_n$-maps, the sets $M_k=\pi_k^n(M_n)$ are also dense in $C(\mathbb B^k,X)$ and consist of homological $Z_n$-maps.

 $(2) \Rightarrow (3)$: If $z\in H_k(U,V)$ for some open sets $U,V$ then there is a singular chain $c_z=\sum_{i=1}^{p}m_if_i$ representing $z$ with $f_i(\mathbb S^{k-1})\subset V$ for all $i$. Using the density of $M_k$ in  $C(\mathbb B^k,X)$ we approximate each $f_i$ with a map $g_i\in M_k$ such that $g_i(\mathbb B^k)\subset U$ and $g_i(\mathbb S^{k-1})\subset V$. Because $X\in LC^n$, we can suppose that each $g_i$ is homotopic to $f_i$ in $U$. This means that $\sum_{i=1}^{p}m_ig_i$ is another representation of $z$ and
$C_z=\bigcup_{i=1}^p g_i(\mathbb B^k)$ is a singular carrier for $z$. Because each $g_i(\mathbb B^k)$ is a homological $Z_n$-set in $X$, so is $C_z$, see \cite{bck}. It remains to show that $C_z$ is also a \v{C}ech carrier for $z$. And this follows from Lemma \ref{carrier}.

$(3)\Rightarrow  (4)$: Let $z_j\in H_{k(j)}(U_j,V_j)$, $j=1,2$, and $(C_1,\partial C_1)\subset (U_1,V_1)$ be a \v{C}ech carrier for $z_1$ such that $C_1$ is a homological $Z_n$-set. Then, by Lemma \ref{1}, $z_2$ has a singular carrier $(C_2,\partial C_2)\subset (U_2,V_2)$ with $C_1\cap C_2=\varnothing$. Lemma \ref{carrier} implies that $(C_2,\partial C_2)$ is also a
\v{C}ech carrier for $z_2$.
\end{proof}

\begin{lem}\label{G-delta}
Let $X$ be an $LC^n$ space. Then the set $\Lambda _n$ of all homological $Z_n$-maps $f\colon\mathbb B^n \to X$ is $G_{\delta}$ in $C(\mathbb B^n,X)$.
\end{lem}

\begin{proof}
Let $\mathcal B$ be a countable finitely additive base for $X$.  Let $\mathcal C_k$ be a countable dense set of $k$-chains in $X$, where  $k=0,1,2,\dots, n$ (density here is with respect to the compact-open topology for respective maps), and $\mathcal C=\bigcup_{k=0}^n\mathcal C_k$. For a set $A\subset X$ and $\varepsilon >0$  let $B (A,\epsilon)$ denote the closed $\varepsilon$-neighbourhood of $A$ in $X$.
 For any  $U\in\mathcal B$, $c\in\mathcal C$ with $c\subset U$, and $m=1,2, \dots$ let $G^m_{U, c}$ be the set of all maps $f\colon\mathbb B^n \to X$ such that if $\partial c \subset U\setminus B(f(\mathbb B^n),1/m)$ then there exists a $(k+1)$-chain $c'\subset U$ such that $\partial c' - c\subset U\setminus f(\mathbb B^n)$.

First we will show that $\Lambda _n = \cap \{ G^m_{U,c} \mid  U\in \mathcal B, c\in\mathcal C, m\in\mathbb N\}$. Denote the latter intersection by $\mathcal G$ and suppose that $f\in \mathcal G$. According to the results of \cite{bck} it is sufficient to show that $H_k (U, U\setminus f(\mathbb B^n))=0$ for all $k=0,1,2,\dots, n$ and all $U\in \mathcal B$. Consider a chain $c\in  H_k (U, U\setminus f(\mathbb B^n))$. We may assume that $c\in \mathcal C_k$.  Because the carrier of $c$ and $f(\mathbb B^n)$ are compact, there exists $m$ such that $\partial c \subset U\setminus B(f(\mathbb B^n),1/m))$. Since $f\in G^m_{U, c}$ there exists a $(k+1)$-chain $c'\subset U$ such that $\partial c' - c\subset U\setminus f(\mathbb B^n)$. This implies that $c$ is homologous to $0$ in $H_k (U, U\setminus f(\mathbb B^n)$. Thus $f\in \Lambda _n$. Now consider any $f\in \Lambda _n$ and a set $G^m_{U, c}$ with $c\in\mathcal C_k$. Suppose $\partial c \subset U\setminus B(f(\mathbb B^n),1/m)$. Then, in particular, $c\in H_k (U, U\setminus f(\mathbb B^n)) =0$. Therefore there exists  $(k+1)$-chain $c'\subset U$ such that $\partial c' - c\subset U\setminus f(\mathbb B^n)$. This implies that $f$ is contained in each set $G^m_{U, c}$, and hence $f\in \mathcal G$.

 Next, we will show that each $G^m_{\mathcal U,c}$ is open. For this, consider $f\in G^m_{\mathcal U,c}$. If $\partial c \not\subset U\setminus B(f(\mathbb B^n),1/m)$, the same is true for all maps $g$ that are sufficiently close to $f$. Suppose that $\partial c \subset U\setminus B(f(\mathbb B^n),1/m)$.
  Then there exists a $(k+1)$-chain $c'$ such that $\partial c' - c\subset U\setminus f(\mathbb B^n)$. This implies that $\partial c' - c\subset U\setminus g(\mathbb B^n)$ for all maps $g$ sufficiently close to $f$.
\end{proof}

We say that a space $X$ has the {\em property $\displaystyle\rm{DD^{\{n,m\}}P}$} if every two maps $f:\mathbb B^n\to X$ and $g:\mathbb B^m\to X$ can be approximated by maps $f':\mathbb B^n\to X$ and $g':\mathbb B^m\to X$ with $f'(\mathbb B^n)\cap g'(\mathbb B^m)=\varnothing$.
The property $\displaystyle\rm{DD^{\{1,2\}}P}$ (resp., $\displaystyle\rm{DD^{\{1,1\}}P}$)
is called the disjoint arc-disk (resp., disjoint arcs) property.
\begin{pro}\label{dadp}
Let $C(\mathbb B^n,X)$ contains a dense set of homological $Z_1$-maps. Then
\begin{itemize}
\item[(1)] $\displaystyle X\in\rm{DD^{\{1,n\}}P}$;
\item[(2)] $C(\mathbb B^1,X)$ contains a dense $G_\delta$-subset consisting of $Z_n$-maps;
\item[(3)] $C(\mathbb B^n,X)$ contains a dense $G_\delta$-subset consisting of $Z_1$-maps.
\end{itemize}
\end{pro}

\begin{proof}
The following statement was actually established in the proof of \cite[Proposition 6]{br}: If $A$ is a homological $Z_1$-set in a space $X$, then
every map $g:\mathbb B^1\to X$ can be approximated by maps $g':\mathbb B^1\to X\setminus A$. This statement implies that $\displaystyle X\in\rm{DD^{\{1,n\}}P}$
provided $C(\mathbb B^n, X)$ contains a dense set of homological $Z_1$-maps.

To prove the second item, choose a countable base $\{U_i\}$ for  $C(\mathbb B^n,X)$. For every $i$ let $G_i$ be the set of all $f\in C(\mathbb B^1,X)$ such that $f(\mathbb B^1)\cap g(\mathbb B^n)=\varnothing$ for some
$g\in U_i$. Since $\displaystyle X\in\rm{DD^{\{1,n\}}P}$, one can show that each $G_i$ is open and dense in $C(\mathbb B^1,X)$. So, $G=\bigcap G_i$ is dense and $G_\delta$ in $C(\mathbb B^1,X)$. Observe that for every $f\in G$ and every $i$ there is $g_i\in U_i$ with $f(\mathbb B^1)\cap g_i(\mathbb B^n)=\varnothing$. Since $\{g_i\}$ is a dense set in $C(\mathbb B^n,X)$, each $f(\mathbb B^1)$, $f\in G$, is a $Z_n$-set in $X$. The proof of item $(3)$ is similar.
\end{proof}

%Note that item $(3)$ from Proposition \ref{main} implies that each point of $x$ is a homological $Z_n$-set.
\begin{pro}\label{dd^np}
Let $X$ be a locally compact $LC^n$-space with $n\geq 2$. Then the following are equivalent:
\begin{itemize}
\item[(1)] $C(\mathbb B^n,X)$ contains a dense set of homological $Z_n$-maps;
\item[(2)] $X\in\rm{DC^nP}$ and $C(\mathbb B^2,X)$ contains a dense set of homological $Z_n$-maps;
\item[(3)] Every $C(\mathbb B^k,X)$, $k\leq n$, contains a dense $G_\delta$-set of homological $Z_n$-maps.
\end{itemize} %such that $C(\mathbb B^n,X)$ contains a dense set of homological $Z_n$-maps.
%Then $C(\mathbb B^k,X)$ contains a dense $G_\delta$-set of homological $Z_n$-maps for every $k\leq n$.
%If, in addition, $X\in\rm{DD^2P}$, then $X\in\rm{DD^nP}$.
\end{pro}

\begin{proof}
$(1)\Rightarrow (2)$: This implication follows from Proposition \ref{main}.

$(2)\Rightarrow (3)$: Suppose $X$ satisfies condition $(2)$ and choose a dense sequence $\{g_j\}$ in $C(\mathbb B^2,X)$ of homological $Z_n$-maps.
 Let $\mathcal H_n=\{z_j\in H_{k(j)}(U_j,V_j):k(j)\leq n{~}\mbox{and}{~}U_j,V_j\in\mathcal B\}$, where $\mathcal B$ is a countable additive base for $X$. Then the arguments from the proof of \cite[Lemma 3.2]{dw} imply that every
 $z_j\in H_{k(j)}(U_j,V_j)$ has a homological \v{C}ech $Z_n$-carrier $(C_j,\partial C_j)\subset (U_j,V_j)$.
%Because any $z\in H_k(U,V)$, $k\leq n$, has a homological $Z_n$-carrier $(C_z,\partial C_z)\subset (U,V)$ (see Proposition \ref{main}),
%every $z_j\in\mathcal H_n$ has a homological $Z_n$-carrier $(C_j,\partial C_j)\subset (U_j,V_j)$.
According to Corollary \ref{hom}, every compact subset of $X\setminus\bigcup C_j$ is a homological $Z_n$-set in $X$. Therefore, we have a sequence
$\{D_j=C_j\cup g_j(\mathbb B^2)\}$ of homological $Z_n$-sets such that every compact subset of $X\setminus\bigcup D_j$ is also a homological $Z_n$-set. Moreover, the density of $\{g_j\}$ in $C(\mathbb B^2,X)$ implies that all compact subsets of $X\setminus\bigcup D_j$ are $Z_2$-sets. So, by Proposition \ref{z}, every compact subset of $X\setminus\bigcup D_j$ is a $Z_n$-set. Following the arguments of \cite[Lemma 3.8]{dw}, one can show that there is another sequence $\{A_i\}\subset X\setminus\bigcup D_j$ of compact sets with each compact subset of $X\setminus\bigcup A_i$ being a homological $Z_n$-set. Because $A_i$ are $Z_n$-sets (as subsets of $X\setminus\bigcup D_j$), for every $k\leq n$ the maps $g\in C(\mathbb B^k,X)$ with
 $g(\mathbb B^k)\cap(\bigcup A_i)=\varnothing$ form a dense $G_\delta$-subset $W_k$ of $C(\mathbb B^k,X)$. Finally, observe that $g(\mathbb B^k)$ is a homological $Z_n$-set for every $g\in W_k$.

$(3)\Rightarrow (1)$: This implication is obvious.
\end{proof}

 \begin{thm}\label{t1}
The following conditions are equivalent for any  $LC^n$-space $X$:
\begin{itemize}
\item[(1)] $X\in DD^2P$ and $C(\mathbb B^n,X)$ contains a dense set $\Lambda_n$ of homological $Z_n$-maps;
\item[(2)] $X$ has the disjoint $n$-disks property.
\end{itemize}
%Moreover, if $X$ is locally compact, the requirement $\Lambda_n$ to be $G_\delta$ in $C(\mathbb B^n,X)$ can be omitted.
\end{thm}
\begin{proof}
Suppose $X$ satisfies condition $(1)$.
Since $X\in\rm{DD^2P}$, there is a dense $G_\delta$-set $G'\subset C(\mathbb B^2,X)$ of $Z_2$-maps. Observe that there is a dense sequence $\{g_i\}\subset C(\mathbb B^2,X)$ such that each $g_i(\mathbb B^2)$ is a $Z_2$-set and a homological $Z_n$-set. Indeed, because the restriction map  $\pi^n_2:C(\mathbb B^n,X)\to C(\mathbb B^2,X)$ is surjective and open, the set $(\pi^n_2)^{-1}(G')$ is dense and $G_\delta$ in $C(\mathbb B^n,X)$.
%On the other hand, $C(\mathbb B^n,X)$ contains a dense $G_\delta$-set $\Lambda_n$ consisting of homological $Z_n$-maps.
Proposition \ref{G-delta} implies that $\Lambda_n$ is $G_{\delta}$ in $C(\mathbb B^n, X)$.
So,  $(\pi^n_2)^{-1}(G')\cap\Lambda_n$ is also dense
 in $C(\mathbb B^n,X)$ and it contains a dense sequence $\{f_i\}$. Obviously, the sequence $\{g_i=\pi^n_2(f_i)\}$ has the required property. Thus, by Proposition \ref{z}, all $g_i(\mathbb B^2)$ are $Z_n$-sets. Consequently, the set
 $\Gamma_n=\{f\in C(\mathbb B^n,X): f(\mathbb B^n)\cap(\bigcup g_i(\mathbb B^2))=\varnothing\}$ is dense and $G_\delta$ in $C(\mathbb B^n,X)$.
 Moreover, the density of $\{g_i\}$ in $C(\mathbb B^2,X)$ implies that $f(\mathbb B^n)$ is a $Z_2$-set in $X$ for all $f\in\Gamma_n$.
 Then  $\Gamma_n\cap\Lambda_n$ is a dense subset of $C(\mathbb B^n,X)$ and consists of maps $f$ such that $f(\mathbb B^n)$ is both a homological $Z_n$-set and a $Z_2$-set in $X$. Therefore, each $f\in\Gamma_n\cap\Lambda_n$ is a $Z_n$-map, which yields that $X$ has the disjoint $n$-disks property.

 %Because  $X\in\rm{DD^nP}$ implies $X\in DD^2P$ and the existence of a dense $G_\delta$-subset of $C(\mathbb B^n,X)$ consisting of $Z_n$-maps,
 The implication $(2)\Rightarrow (1)$ is obvious.
 \end{proof}

Following Bowers \cite{bo}, we say that a space $X$ has a {\em nice $ANR$ local compactification} if there is a locally compact $ANR$-space $Y$ containing $X$ such that $X=Y\setminus F$ for some $Z_\sigma$-set $F$ (i.e., a countable union of $Z$-sets) in $Y$. Any such $X$ is complete $ANR$, see \cite{to2}.
Toru\'{n}czyk's \cite{to1} characterization theorem of $Q$-manifolds yields the following proposition (the special case when $F=\varnothing$ was established in \cite{br}):
\begin{pro}\label{t}
Let $\overline X$ be a nice $ANR$ local compactification of a space $X$. Then $\overline X$ is a $Q$-manifold if and only if $X$ has  the disjoint disks property and for every $n$ the space $C(\mathbb B^n,X)$ contains a dense set of homological $Z_n$-maps.
\end{pro}

\begin{proof}
Let $X=\overline X\setminus F$, where $F$ is a $\sigma Z$-set in $\overline X$ (i.e., $F$ is the  union of countably many $Z$-sets).
Suppose $X$ has  the disjoint disks property and every $C(\mathbb B^n,X)$ contains a dense set of homological $Z_n$-maps.
To show that $\overline X$ is a $Q$-manifold, according to Theorem \ref{t1} and Toru\'{n}czyk's \cite{to1} characterization theorem of $Q$-manifolds, it suffices to prove $\overline X$ satisfies the following two conditions: $(i)$ $\overline X$ has the disjoint disks property and $(ii)$ every
$C(\mathbb B^n,\overline X)$ contains a dense set of homological $Z_n$-maps. Let $f,g:\mathbb B^2\to\overline X$ be two maps. Since $F$ is a $\sigma Z$-set in $\overline X$, we can approximate $f,g$, respectively, by maps $f',g':\mathbb B^2\to X$. Then, using that $X\in\rm{DD^2P}$, approximate $f',g'$ by maps $f'',g'':\mathbb B^2\to X$ with $f''(\mathbb B^2)\cap g''(\mathbb B^2)=\varnothing$. To show condition $(ii)$ , let $f\in C(\mathbb B^n,\overline X)$. Since $C(\mathbb B^n,X)$ contains a dense set of homological $Z_n$-maps, and using again that $F$ is $\sigma Z$-set in $\overline X$, we can suppose that $f(\mathbb B^n)$ is a homological $Z_n$-set in $X$. It remains to show that
$f(\mathbb B^n)$ is a homological $Z_n$-set in $\overline X$. To this end, by
Lemma \ref{1}, it suffices to show that any  $z\in H_m(U,V)$, where $m\leq n$ and $U,V$ are open sets in $\overline X$, has a singular carrier disjoint from $f(\mathbb B^n)$. If
$\sum_{i=1}^k m_ih_i$ is a singular representation of $z$ with $h_i\in C(\mathbb B^n,\overline X)$, we approximate
each $h_i$ by a map $h_i'\in C(\mathbb B^n,X)$ such that $h'(\mathbb B^n,\mathbb S^{n-1})\subset (U\cap X,V\cap X)$ and $h'$ is homotopic to $h$ in $\overline X$. Therefore, we may assume that  $h_i\in C(\mathbb B^n,X)$ for all $i$ and $z\in H_m(U\cap X,V\cap X)$. Since  $f(\mathbb B^n)$ is a homological $Z_n$-set in $X$, by Lemma \ref{1}, there exists a singular carrier $(C_z,\partial C_z)\subset (U\cap X,V\cap X)$ of $z$ with
$C_z\cap f(\mathbb B^n)=\varnothing$.

If $\overline X$ is a $Q$-manifold, then it has the disjoint $n$-disks property for every $n$ \cite{to1}. Since $F$ is a $\sigma Z$-set in $\overline X$, this implies that $X$ also has the disjoint $n$-disks property for every $n$. Equivalently, each function space $C(\mathbb B^n,X)$ contains a dense set of $Z_n$-maps. Because every $Z_n$-map is a homological $Z_n$-map, the proof is completed.
\end{proof}

\begin{thm}\label{q-mfd}
Let $X$ be a locally compact $LC^n$-space such that $C(\mathbb B^n,X)$ contains a dense set of homological $Z_n$-maps.
Then $X\times Y$ has the disjoint $n$-disks property for every non-trivial locally compact $ANR$-space $Y$.
\end{thm}

\begin{proof}
%Suppose for every $n$ the space $C(\mathbb B^n,X)$ contains a dense set of homological $Z_n$-maps.
According to Proposition \ref{dadp}, $X$ has the disjoint arc-disk property. Then, following the proof of \cite[Proposition 2.10]{d}, one can show that $X\times Y$ has the disjoint disks property. Since $X\times Y$ is $LC^n$,
by Theorem \ref{t1}, it suffices to show that
 $C(\mathbb B^n,X\times Y)$  contains a dense set of homological $Z_n$-maps. To this end, let $f=(f_1,f_2)\in C(\mathbb B^n,X\times Y)$, where
$f_1\in C(\mathbb B^n,X)$ and $f_2\in C(\mathbb B^n,Y)$.
We can approximate $f_1$ by maps
$g\in C(\mathbb B^n,X)$ such that $g(\mathbb B^n)$ are homological $Z_n$-sets in $X$. Then, for any such $g$ consider the map
$h_g=(g,f_2):\mathbb B^n\to X\times Y$. Obviously, the maps $h_g$ approximate $f$ and $h_g(\mathbb B^n)\subset g(\mathbb B^n)\times Y$. It remains to show that $g(\mathbb B^n)\times Y$ is a homological $Z_n$-set in  $X\times Y$.
Indeed, let $\mathcal B_X$ and $\mathcal B_Y$ be bases for $X$ $Y$, respectively. Then, by \cite[Proposition 3.6]{bck}, it suffices to show that for any $U\in\mathcal B_X$ and $V\in\mathcal B_Y$  we have
$H_k(U\times V,(U\times V)\setminus(g(\mathbb B^n)\times Y))=0$ for all $k\leq n$. And this is really true because  by the K\"{u}nneth formula
the group $H_k(U\times V,(U\times V)\setminus(g(\mathbb B^n)\times Y))$ is isomorphic to the direct sum of
$\sum_{i+j\leq k}H_i(U,U\setminus g(\mathbb B^n))\otimes H_j(V)$ and  $\sum_{i+j\leq k-1}H_i(U,U\setminus g(\mathbb B^n))* H_j(V)$, where
$H_i(U,U\setminus g(\mathbb B^n))\otimes H_j(V)$ and $H_i(U,U\setminus g(\mathbb B^n))* H_j(V)$ stand for the tensor and torsion products of
$H_i(U,U\setminus g(\mathbb B^n))$ and $H_j(V)$.
\end{proof}

Another implication of Toru\'{n}czyk's \cite{to1} $Q$-manifolds characterization theorem provides next corollary.
\begin{cor}\label{anr}
Let $X$ be a locally compact $ANR$ such that for every $n$ the space $C(\mathbb B^n,X)$ contains a dense set of homological $Z_n$-maps.
Then $X\times Y$ is a $Q$-manifold for every non-trivial locally compact $ANR$-space $Y$.
\end{cor}
A similar statement was established in \cite[Theorem 14]{br}.

Corollary \ref{anr} implies that $X\times\mathbb B^1$ is a $Q$-manifold provided $X$ is a locally compact $ANR$ such that for every $n$ the space
$C(\mathbb B^n,X)$ contains a dense set of homological $Z_n$-maps.
We say that a space $X$ is a {\em fake $Q$-manifold} if $X\times\mathbb B^1$ is a $Q$-manifold, but $X$ is not a $Q$-manifold. According to \cite{dw}, any fake $Q$-manifold has the disjoint \v{C}ech carrier property but not the disjoint disks property.
All existing examples (see, \cite{br}, \cite{dw} and \cite{s}) of fake $Q$-manifolds $X$ have the property that for any $n$ the space $C(\mathbb B^n,X)$ contains a dense set of homological $Z_n$-maps. So, the following question is very natural.

\begin{qu}\label{q}
Let $X$ be a locally compact $ANR$. Is it true that $X\times\mathbb B^1$ is a $Q$-manifold if and only if for every $n$ the space $C(\mathbb B^n,X)$ contains a dense set of homological $Z_n$-maps$?$
\end{qu}

According to next proposition, Question \ref{q} has a positive solution if we can show that for every $n\geq 2$ the space $C(\mathbb B^2,X)$ contains a dense set of homological $Z_n$-maps provided $X\times\mathbb B^1$ is a $Q$-manifold. In particular, that would be true if $C(\mathbb B^2,X)$ contains a dense subset of maps with finite-dimensional images.

\begin{pro}
Let $X$ be a locally compact $ANR$ such that for every $n$ the space $C(\mathbb B^2,X)$ contains a dense set of homological $Z_n$-maps. Then
$X\times\mathbb B^1$ is a $Q$-manifold if and only if each $C(\mathbb B^n,X)$, $n\geq 2$, contains a dense set of homological $Z_n$-maps.
\end{pro}

\begin{proof}
If $X\times\mathbb B^1$ is a $Q$-manifold, then so is $X\times\mathbb B^2$. Hence, by \cite[Corollary 6.2]{dw}, $X\in\rm{DC^nP}$.
This, according to Proposition \ref{dd^np}, implies that every $C(\mathbb B^n,X)$ contains a dense set of homological $Z_n$-maps.
The other implication follows from Theorem \ref{q-mfd}.
\end{proof}

According to \cite{dh}, the so called {\em disjoint path concordance property} characterizes locally compact $ANR$s $X\in\rm{DD^{\{1,1\}}}P$ such that
$X\times\mathbb R$, or equivalently $X\times\mathbb B^1$, has the disjoint disks property. The disjoint path concordance property is quite different from the property that for every $n$ the space $C(\mathbb B^n,X)$ contains a dense set of homological $Z_n$-maps, but
%We say that a metric space $(X,\rho)$ satisfies the
%disjoint path concordance property if for any two homotopies $f_1,f_2:\mathbb B^1\times [0,1]\to X$ and any $\varepsilon>0$, there exists path homotopies $F_1,F_2:\mathbb B^1\times [0,1]\to X\times [0,1]$ such that:
%\begin{itemize}
%\item $F_i(\mathbb B^1\times\{\rm{e}\})\subset X\times\{\rm{e}\}$, where $i=1,2$ and $\rm{e}\in\{0,1\}$;
%\item $F_1(\mathbb B^1\times [0,1])\cap F_2(\mathbb B^1\times [0,1])=\varnothing$;
%\item $\rho(f_i,\rm{proj_X}F_i)<\varepsilon$.
%\end{itemize}
the results from \cite{dw} yield the following description of the locally compact $ANR$-spaces $X$ such that $X\times\mathbb B^1$ is a $Q$-manifold.
\begin{pro}
Let $X$ be a locally compact $ANR$-space. Then $X\times\mathbb B^1$ is a $Q$-manifold if and only if $X$ has both the disjoint path concordance property and the disjoint \v{C}ech carrier property.
\end{pro}

%Because of Corollary \ref{anr}, there is another question closely related to Question \ref{q}.

%\begin{qu}\label{qq}
%Let $X$ be a locally compact $ANR$. Is it true that each $C(\mathbb B^n,X)$, $n\geq 2$, contains a dense set of homological $Z_n$-maps
%if and only if $X\times Y$ is a $Q$-manifold for every non-trivial locally compact $ANR$-space $Y$$?$
%\end{qu}

\section{Homological $Z_n$-maps and $l_2$-manifolds}

In this section the function spaces $C(Y,X)$ are equipped with the {\em limitation topology}, see \cite{bo1} and \cite{to3}.
Let $\rm{cov}(X)$ denote the collection of all open covers of $X$.
 For a map $f\in C(Y,X)$ and $\mathcal U\in\rm{cov}(X)$ let $B(f,\mathcal U)$ be the set of maps $g\in C(Y,X)$ that are $\mathcal U$-close to $f$.
A set $U\subset C(Y,X)$ is open in the limitation topology if for every $f\in U$ there exists $\mathcal U\in\rm{cov}(X)$ such that
$B(f,\mathcal U)\subset U$. %Similarly, $U$ is open in the modified limitation topology if for every $f\in U$ there is $\mathcal U\in\rm{cov}(f)$
According to \cite{bo1} and \cite{to3}, $C(Y,X)$ with the limitation topology is a Baire space.

We say $X$ satisfies the {\em discrete $n$-cells property}, where $n\leq\infty$, if for each map $f:\oplus_{i=1}^\infty\mathbb B^n_i\to X$ of the countable free union of
$n$-cells ($\infty$-cells are Hilbert cubes $Q_i$) into $X$ and each open cover $\mathcal U$ of $X$ there exists a map
$g:\oplus_{i=1}^\infty\mathbb  B^n_i\to X$ such that $g$ is $\mathcal U$-close to $f$ and $\{g(\mathbb B^n_i)\}_{i=1}^\infty$ is a discrete family in $X$. The discrete $\infty$-cells property is usually called the {\em discrete approximation property}.
%Toru\'{n}czyk's characterization \cite{to3} of $l_2$-manifolds is well known: an $ANR$ space is an $l_2$-manifold if and only if it satisfies the discrete approximation property.

%Following Bowers \cite{bo} we say that a space $X$ has a nice $ANR$ local compactification if $X=\overline X\setminus F$ for some $\sigma Z$-set $F$ in a locally compact space $\overline X$. According to \cite{to2}, any such $X$ is a complete $ANR$space.
%According to \cite{bo}, $X$ satisfies the discrete carrier property if for every open cover $\mathcal U$ of $X$ and any sequence $\{z_i\in H_{k(i)}(U_i,V_i)\}_{i=1}^\infty$, where $V_i\subset U_i$ are open in $X$ and $k(i)$ any integer, there exists for each $i$ a singular carrier $(C_i,\partial C_i)\subset (\rm{st}(U_i,\mathcal U),\rm{st}(V_i,\mathcal U))$ for the inclusion-induced homomorphism $i_{*}(z_i): H_{k(i)}(U_i,V_i)\to H_{k(i)}(\rm{st}(U_i,\mathcal U),\rm{st}(V_i,\mathcal U))$, such that $\{C_i\}_{i=1}^\infty$forms a discrete family in $X$.

Our first result in this section  provides a homological characterization  of $l_2$-manifolds in the boundary set setting.

\begin{thm}\label{l2manifold}
Suppose $X$ has a nice $ANR$ local compactification. Then $X$ is an $l_2$-manifold if and only if $X$ satisfies the following conditions:
\begin{itemize}
\item[(1)] $X$ has the disjoint disks property;
\item[(2)] For every $n\geq 2$ the space $\displaystyle C(\oplus_{i=1}^\infty\mathbb B^n_i,X)$,  equipped with the limitation topology, contains a dense $G_\delta$-set of maps $f$ such that the set
$\displaystyle f(\oplus_{i=1}^\infty\mathbb B^n_i)\subset X$ is closed and each $f(\mathbb B^n_i)$ is a homological $Z_n$-set in $X$.
\end{itemize}
\end{thm}

\begin{proof}
Let $\overline X$ be a locally compact $ANR$-compactification of $X$ such that $X=\overline X\setminus F$, where $F$ is a $\sigma Z$-set in
$\overline X$. Suppose $X$ satisfies conditions $(1)$ and $(2)$. Observe that, by condition $(2)$, each $C(\mathbb B^n,X)$ contains a dense set of homological $Z_n$-maps. Therefore, by Proposition \ref{t},
$\overline X$ is a $Q$-manifold.

\begin{claim}
Every $C(\mathbb B^n,X)$ contains a dense $G_\delta$-set of $Z_n$-maps.
\end{claim}
According to \cite{to1}, a given space $Y$ has the disjoint $n$-disks property iff all $Z_n$-maps in $C(\mathbb B^n,Y)$ form a dense and $G_\delta$-subset. Since $\overline X$ is a $Q$-manifolds, it has the disjoint $n$-disks property for every $n$ \cite{to1}, or equivalently, every $C(\mathbb B^n,\overline X)$ contains a dense set of $Z_n$-maps. Let $\{g_i\}_{i\geq 1}$ be a dense in $C(\mathbb B^n,\overline X)$ sequence
of $Z_n$-maps and
$D=\bigcup_{i\geq 1}g_i(\mathbb B^n)\cup F$. Then $D$ is a $\sigma Z_n$-subset of $\overline X$. Hence, every compactum in $X\setminus D$ is a $Z_n$-set in $\overline X$ and every map $f\in C(\mathbb B^n,\overline X)$ can be approximated by maps into $X\setminus D$. This implies that  $X$ has the disjoint $n$-disks property and $C(\mathbb B^n,X)$ contains a dense $G_\delta$-set of $Z_n$-maps.

\smallskip
%Condition $(1)$ implies that $X$ can not be locally compact at any point.
For every $n$ we fix a dense $G_\delta$-subset $\Lambda^n\subset\displaystyle C(\oplus_{i=1}^\infty\mathbb B^n_i,X)$ consisting of maps $f$ satisfying the condition
\begin{itemize}
\item [$(*)_n$]: $\displaystyle f(\oplus_{i=1}^\infty\mathbb B^n_i)\subset X$ is closed and each $f(\mathbb B^n_i)$ is a homological $Z_n$-set in $X$.
\end{itemize}
Then, for $k\geq 2$ the space $\displaystyle C(\oplus_{i=k}^\infty\mathbb B^n_i,X)$ contains a dense $G_\delta$-set $\Lambda_k^n$ of maps $f$ such  that
$\displaystyle f(\oplus_{i=k}^\infty\mathbb B^n_i)\subset X$ is closed and all $f(\mathbb B^n_i)$, $i\geq k$, are homological $Z_n$-sets in $X$.
Since each restriction map $\displaystyle p^n_k:C(\oplus_{i=1}^\infty\mathbb B^n_iX)\to C(\oplus_{i=k}^\infty\mathbb B^n_i,X)$ is open (see \cite{to3}) and surjective, all sets $(p^n_k)^{-1}(\Lambda^n_k)$, $k\geq 2$, are dense and $G_\delta$ in $C(\oplus_{i=1}^\infty\mathbb B^n_i,X)$. So is the set $\widetilde\Lambda^n=\bigcap_{k\geq 2}\Lambda^n\cap (p^n_k)^{-1}(\Lambda^n_k)$.
%(recall that $C(\oplus_{i=1}^\infty\mathbb B^n_i,X)$ with the limitation topology has the Baire property, see \cite{bo1}, \cite{to3}).

%\smallskip
%Let $\pi_k^n: C(\oplus_{i=1}^\infty\mathbb B^n_i,X)\to C(\mathbb B^n_k,X)$, $k\geq 1$, denote the restriction map, and $\mathcal Z_k(n)$ be the set of all
%$Z_n$-maps from $\mathbb B^n_k$ to $X$. We already observed that $\mathcal Z_k(n)$ is dense and $G_\delta$ in $C(\mathbb B^n_k,X)$.
%So, $\mathcal Z(n)=\bigcap_{k=1}^\infty(\pi_k^n)^{-1}(\mathcal Z_k(n))$ is dense and $G_\delta$-subset of
%$C(\oplus_{i=1}^\infty\mathbb B^n_i,X)$.

\smallskip
For every $k\neq l$ let $\Lambda^n_{kl}$ be the set of all maps $f\in C(\oplus_{i=1}^\infty\mathbb B^n_i,X)$ such that
$f(\mathbb B^n_k)\cap f(\mathbb B^n_l)=\varnothing$. Obviously, each $\Lambda^n_{kl}$ is open in $C(\oplus_{i=1}^\infty\mathbb B^n_i,X)$.
%$\Lambda^n_{kl}\cap$ is also open in $\widetilde\Lambda^n$.
\begin{claim}
Every $\Lambda^n_{kl}$ is dense in $C(\oplus_{i=1}^\infty\mathbb B^n_i,X)$.
\end{claim}

Indeed, let $f\in C(\oplus_{i=1}^\infty\mathbb B^n_i,X)$. Then the maps $f_k=f|\mathbb B^n_k$ and $f_l=f|\mathbb B^n_l$ can be approximated,
respectively, by maps $f_k'\in C(B^n_k,X)$ and $f_l'\in C(B^n_l,X)$ such that $f_k'(\mathbb B^n_k)\cap f_l'(\mathbb B^n_l)=\varnothing$.
 This can be done because $C(\mathbb B^n_k,X)$ and $C(\mathbb B^n_l,X)$ contain dense sets of $Z_n$-maps. Define a map $g\in C(\oplus_{i=1}^\infty\mathbb B^n_i,X)$ by $g|\mathbb B^n_k=f_k'$, $g|\mathbb B^n_l=g_l'$ and $g|\mathbb B^n_i=f|\mathbb B^n_i$ for all $i\notin\{k,l\}$.
Then $g$ is an approximation of $f$ and $g\in\Lambda^n_{kl}$.

Therefore, the set $\Gamma^n=\bigcap_{k\neq l}\Lambda^n_{kl}$ is dense and $G_\delta$ in  $C(\oplus_{i=1}^\infty\mathbb B^n_i,X)$.
Consequently, so is the set $\Gamma^n\cap\widetilde\Lambda^n$. Observe that $\Gamma^n\cap\widetilde\Lambda^n$ consists of maps $f$ satisfying the following conditions:
\begin{itemize}
\item $f(\mathbb B^n_k)\cap f(\mathbb B^n_l)=\varnothing$ for all $k\neq l$;
\item $\displaystyle f(\oplus_{i=k}^\infty\mathbb B^n_i)$ is a closed set in $X$ for all $k\geq 1$.
\end{itemize}
The last two conditions yields that the family $\{f(\mathbb B^n_i)\}_{i=1}^\infty$ is discrete in $X$ for all $f\in\Gamma^n\cap\widetilde\Lambda^n$. Hence, $X$ has the discrete $n$-cells property for every $n$, and by \cite{bo2}, $X$ has the discrete
approximation property. Finally, we apply Toru\'{n}czyk's \cite{to3} characterization of $l_2$-manifolds.

If $X$ is an $l_2$-manifold, then for every $n$ there is a dense $G_\delta$-set in $C(\oplus_{i=1}^\infty\mathbb B^n_i,X)$ consisting of closed embeddings, see \cite{to3}. Because every compact subset of an $l_2$-manifold is a $Z$-set, $X$ satisfies conditions $(1)$ and $(2)$ from Theorem 3.1.
\end{proof}

Since every closed subset of a homological $Z_n$-set is also homological $Z_n$-set, we have the following corollary.
\begin{cor}
A space $X$ having a nice $ANR$ local compactification is an $l_2$-manifold if and only if $X$ has the disjoint disks property
and for every $n$ the space $\displaystyle C(\oplus_{i=1}^\infty\mathbb B^n_i,X)$,  equipped with the limitation topology, contains a dense $G_\delta$-set of homological $Z_n$-maps.
\end{cor}

One can show that condition $(2)$ in Theorem 3.1 can be replaced by the following two conditions:
\begin{itemize}
\item[$(2')$] For every $n$ the space $C(\mathbb B^n,X)$ contains a dense set of homological $Z_n$-maps;
\item[$(2'')$] For every $n$ the space $\displaystyle C(\oplus_{i=1}^\infty\mathbb B^n_i,X)$ contains a dense $G_\delta$-set of maps with closed images.
\end{itemize}
\begin{cor}
A space $X$ having a nice $ANR$ local compactification is an $l_2$-manifold if and only if $X$ has the disjoint disks property and satisfies conditions $(2')$ and $(2'')$.
\end{cor}

Following \cite{di} we say that a subset $A\subset X$ is {\em almost strongly negligible} if for every open cover $\mathcal U$ of $X$ there is a homeomorphism $h$ from $X$ onto $X\setminus A$ that is $\mathcal U$-close to the identity of $X$.
We provide another version of Theorem 3.1.
\begin{thm}
Suppose $X$ has a nice $ANR$ local compactification $\overline X$. Then $X$ is an $l_2$-manifold if and only if $X$ satisfies the following conditions:
\begin{itemize}
\item[(1)] $X$ has the disjoint disks property;
\item[(2)] For every $n\geq 2$ the space $C(\mathbb B^n,X)$ contains a dense set of homological $Z_n$-maps;
\item[(3)] Every $Z$-subset of $X$ is almost strongly negligible.
\end{itemize}
\end{thm}

\begin{proof}
It is well known that every $l_2$-manifold satisfies the three conditions. For the inverse implication, observe that $\overline X$ is a $Q$-manifold (see the proof of Theorem~\ref{l2manifold}). Hence, $C(Q,\overline X)$ contains a dense set of $Z$-maps, see \cite{to1}.
This implies that $C(Q,X)$ contains a dense set $\mathcal Z$ of $Z$-maps because $X$ is the complement of an $\sigma Z$-set in $\overline X$, see the proof of Claim~1 in Theorem~\ref{l2manifold}.
We choose a countable dense subset $\{\varphi_k\}_{k\geq 1}$ of $\mathcal Z$, and consider the sets
$$\Gamma_k=\{f\in C(\oplus_{i=1}^\infty Q_i,X): \big(f(\oplus_{i=1}^k Q_i)\cup\varphi_k(Q)\big)\cap\overline{f(\oplus_{i=k+1}^\infty Q_i)}=\varnothing\}, k\geq 1.$$
\begin{claim}
Each $\Gamma_k$, $k\geq 1$, is open and dense in $C(\oplus_{i=1}^\infty Q_i,X)$.
\end{claim}
The openness of $\Gamma_k$ is obvious. To show the density, choose $f\in C(\oplus_{i=1}^\infty Q_i,X)$ and an open cover $\mathcal U$ of $X$. We may assume that all sets $f(Q_i)$, $i\leq k$, are $Z$-sets in $X$. Since every $Z$-set in $X$ is a strong $Z$-set (see \cite{bo2}), so is the set $f(\oplus_{i=1}^k Q_i)\cup\varphi_k(Q)$. Consequently, there is a map
$g\in C(\oplus_{i=k+1}^\infty Q_i,X)$ such that $g$ is $\mathcal U$-close to the restriction map $f|\oplus_{i=k+1}^\infty Q_i$
and $\overline{g(\oplus_{i=k+1}^\infty Q_i)}\cap\big(f(\oplus_{i=1}^k Q_i)\cup\varphi_k(Q)\big)=\varnothing$. Finally, the map
$f'\in C(\oplus_{i=1}^\infty Q_i,X)$, defined by $f'|Q_i=f|Q_i$ for $i\leq k$ and $f'|Q_i=g|Q_i$ for $i\geq k+1$, is $\mathcal U$-close to $f$ and $f'\in\Gamma_k$.

Since $C(\oplus_{i=1}^\infty Q_i,X)$ with the limitation topology is a Baire space, the set $\Gamma=\bigcap_{k\geq 1}\Gamma_k$ is dense in $C(\oplus_{i=1}^\infty Q_i,X)$. Observe that $f\in\Gamma$ implies that $f(Q_i)\cap f(Q_j)=\varnothing$ for all
$i\neq j$ and  $A_f=\overline{f(\oplus_{i=1}^\infty Q_i)}\setminus f(\oplus_{i=1}^\infty Q_i)$ is a closed set in $X$ disjoint from each $\varphi_k(Q)$, $k\geq 1$. Because $\{\varphi_k\}_{k\geq 1}$ is dense in $C(Q,X)$, $A_f$ are $Z$-sets of $X$. We can complete the proof by showing that $X$ has the discrete approximation property. To this end, let $\mathcal U, \mathcal V$ be open covers of $X$ such that $\mathcal V$ is a star-refinement of $\mathcal U$ and $f\in C(\oplus_{i=1}^\infty Q_i,X)$. We first take $g\in\Gamma$ that is $\mathcal V$-close to $f$. Since  $A_g$ is a $Z$-set in $X$, there is a homeomorphism $h$ from $X$ onto $X\setminus A_g$ that is $\mathcal V$-close to the identity of $X$.
Then, $\widetilde f=h\circ g$ is $\mathcal U$-close to $f$ and $\{\widetilde f(Q_i)\}_{i\geq 1}$ is a discrete family in $X$.
\end{proof}

Next theorem is a homological version of Toru\'{n}czyk's \cite{to3} characterization of $l_2$-manifolds.
\begin{thm}
An $ANR$ space $X$ is an $l_2$-manifold if and only if $X$ has the discrete $2$-disks property and $\displaystyle C(\oplus_{i=1}^\infty Q_i,X)$,  equipped with the limitation topology, contains a dense $G_\delta$-set of maps
$f$ satisfying the following condition:
\begin{itemize}
\item[(*)] $\displaystyle f(\oplus_{i=1}^\infty Q_i)$ is closed and each $f(Q_i)$ is a homological $Z_\infty$-set in $X$.
\end{itemize}
\end{thm}

\begin{proof}
Toru\'{n}czyk's characterization of $l_2$-manifolds guarantees that every $l_2$-manifold satisfies the hypothesis of Theorem 3.3. For the other implication, we first observe that every compact subset of $X$ is a $Z_2$-set because $X$ has the discrete $2$-disks property.
%Indeed, fix a compact  set $K\subset X$ and a map $h:\mathbb B^2\to X$. Consider the map $\widetilde h:\oplus_{i=1}^\infty\mathbb B^2_i\to X$ with
%$\widetilde h|\mathbb B^2_i=h$ for all $i$. Then $\widetilde h$ can be approximated by maps $g:\oplus_{i=1}^\infty\mathbb B^2_i\to X$ such that
%$\{g(\mathbb B^2_i)\}_{i\geq 1}$ is a discrete family in $X$. Since $K$ is compact, there is $j$ with $g(\mathbb B^2_j)\cap K=\varnothing$. Therefore, $h$ is approximated by maps whose images don't meet $K$.
Thus, every compact homological $Z_\infty$-subset of $X$ is a $Z$-set in $X$. So, condition $(*)$ implies that each $C(Q_i,X)$ contains a dense $G_\delta$-set of $Z$-maps. Then, proceeding as in the proof of Theorem 3.1, we can show that $C(\oplus_{i=1}^\infty Q_i,X)$ contains a dense subset $\Gamma$ such that
$\{f(Q_i)\}_{i\geq 1}$ is a discrete family in $X$ for all $f\in\Gamma$. Therefore,
$X$ has the discrete approximation property.
\end{proof}

\begin{cor}
An $ANR$ space $X$ is an $l_2$-manifold if and only if $X$ has the discrete $2$-cells property and
$\displaystyle C(\oplus_{i=1}^\infty Q_i,X)$ contains a dense $G_\delta$-set of homological $Z_\infty$-maps.
\end{cor}

\noindent
\textbf{Acknowledgements.}
The authors would like to express their gratitude to J. West for his careful reading and helpful comments.

\end{document}